\documentclass[reqno]{amsart}

\usepackage[latin1]{inputenc}
\usepackage{amssymb}
\usepackage{graphicx}
\usepackage{amscd}
\usepackage[hidelinks]{hyperref}
\usepackage{color}
\usepackage{float}
\usepackage{graphics,amsmath,amssymb}
\usepackage{amsthm}
\usepackage{amsfonts}
\usepackage{latexsym}
\usepackage{epsf}
\usepackage{enumerate}
\usepackage{xifthen}
\usepackage{mathrsfs}
\usepackage{dsfont}
\usepackage{makecell}
\usepackage[FIGTOPCAP]{subfigure}
\usepackage{amsmath}
\allowdisplaybreaks[4]
\usepackage{listings}
\usepackage{etoolbox}
\usepackage{fancyhdr}

 \newtheoremstyle{mytheorem}
 {3pt}
 {3pt}
 {\slshape}
 {}
 {\bfseries}
 {.}
 { }
 {}

\numberwithin{equation}{section}

\theoremstyle{theorem}
\newtheorem{theorem}{Theorem}[section]

\newtheorem{lemma}[theorem]{Lemma}

\theoremstyle{definition}

\newtheorem{example}{Example}[section]

\newtheorem{question}{Question}[section]

\newtheorem{remark}{Remark}[section]

\newcommand{\doi}[1]{\href{http://dx.doi.org/#1}{doi: #1}}

\newcommand{\Keywords}[1]{\ifthenelse{\isempty{#1}}{}{\smallskip \smallskip \noindent \textbf{Keywords}. #1}}
\newcommand{\MSC}[2][2010]{\ifthenelse{\isempty{#2}}{}{\smallskip \smallskip \noindent \textbf{#1MSC}. #2}}
\newcommand{\abstractnote}[1]{\ifthenelse{\isempty{#1}}{}{\smallskip \smallskip \noindent \textsuperscript{\dag}#1}}

\makeatletter
\def\specialsection{\@startsection{section}{1}%
  \z@{\linespacing\@plus\linespacing}{.5\linespacing}%
  {\normalfont}}
\def\section{\@startsection{section}{1}%
  \z@{.7\linespacing\@plus\linespacing}{.5\linespacing}%
  {\normalfont\scshape}}
\patchcmd{\@settitle}{\uppercasenonmath\@title}{\Large\boldmath}{}{}
\patchcmd{\@settitle}{\begin{center}}{\begin{flushleft}}{}{}
\patchcmd{\@settitle}{\end{center}}{\end{flushleft}}{}{}
\patchcmd{\@setauthors}{\MakeUppercase}{\normalsize}{}{}
\patchcmd{\@setauthors}{\centering}{\raggedright}{}{}
\patchcmd{\section}{\scshape}{\large\bfseries\boldmath}{}{}
\patchcmd{\subsection}{\bfseries}{\bfseries\boldmath}{}{}
\renewcommand{\@secnumfont}{\bfseries}
\patchcmd{\@startsection}{\@afterindenttrue}{\@afterindentfalse}{}{}
\patchcmd{\abstract}{\leftmargin3pc}{\leftmargin1pc}{}{}

\def\maketitle{\par
  \@topnum\z@ 
  \@setcopyright
  \thispagestyle{empty}
  \ifx\@empty\shortauthors \let\shortauthors\shorttitle
  \else \andify\shortauthors
  \fi
  \@maketitle@hook
  \begingroup
  \@maketitle
  \toks@\@xp{\shortauthors}\@temptokena\@xp{\shorttitle}%
  \toks4{\def\\{ \ignorespaces}}
  \edef\@tempa{%
    \@nx\markboth{\the\toks4
      \@nx\MakeUppercase{\the\toks@}}{\the\@temptokena}}%
  \@tempa
  \endgroup
  \c@footnote\z@
  \@cleartopmattertags
}
\makeatother

\newcommand{\qbinom}[2]{\begin{bmatrix}#1\\#2\end{bmatrix}}

\title[An identity of Andrews and Yee]{Combinatorial proof of an identity of Andrews and Yee}

\author[S. Chern]{Shane Chern}
\address{Department of Mathematics, The Pennsylvania State University, University Park, PA 16802, USA}
\email{shanechern@psu.edu; chenxiaohang92@gmail.com}

\date{}

\begin{document}

{\footnotesize\noindent To appear in \textit{Ramanujan J.}\\
\doi{10.1007/s11139-018-0012-0}}

\bigskip \bigskip

\maketitle

\begin{abstract}

Recently, Andrews and Yee studied two-variable generalizations of two identities involving partition functions $p_\omega(n)$ and $p_\nu(n)$ introduced by Andrews, Dixit and Yee. In this paper, we present a combinatorial proof of an interesting identity in their work.


\MSC{Primary 05A19; Secondary 05A17, 11P81.}
\end{abstract}

\section{Introduction}

Recently, Andrews and Yee \cite{AY2017} studied two-variable generalizations of two identities involving partition functions $p_\omega(n)$ and $p_\nu(n)$ introduced by Andrews et al.~\cite{ADY2015}, where $p_\omega(n)$ enumerates the number of partitions of $n$ in which all odd parts are less than twice the smallest part, and $p_\nu(n)$ enumerates the number of partitions of $n$ counted by $p_\omega(n)$ with all parts being distinct. Their results are
\begin{align}
\sum_{n\ge 1} \frac{q^n}{(zq^n;q)_{n+1}(zq^{2n+2};q^2)_\infty}&= \sum_{n\ge 0}\frac{z^n q^{2n^2+2n+1}}{(q;q^2)_{n+1}(zq;q^2)_{n+1}},\label{eq:AY1}\\
\sum_{n\ge 0}q^n (-zq^{n+1};q)_{n}(-zq^{2n+2};q^2)_\infty &=\sum_{n\ge 0} \frac{z^n q^{n^2+n}}{(q;q^2)_{n+1}},\label{eq:AY2}
\end{align}
where we adopt the standard $q$-series notations
$$(a;q)_n:=\prod_{k=0}^{n-1} (1-a q^k)\quad\text{and}\quad (a;q)_\infty:=\prod_{k=0}^{\infty} (1-a q^k).$$

To prove the two identities, Andrews and Yee introduced the sum
$$S_n(i):=\sum_{s=0}^n \frac{q^{is}(q;q)_{n+s}}{(q^2;q^2)_s},$$
where $n$ and $i$ are nonnegative integers. Now the following surprising identity
\begin{equation}\label{eq:AY3}
S_n(1)=\sum_{s=0}^n \frac{q^{s}(q;q)_{n+s}}{(q^2;q^2)_s}=(q^2;q^2)_n
\end{equation}
plays an important role in proving \eqref{eq:AY1} and \eqref{eq:AY2}.

Andrews and Yee's proof of \eqref{eq:AY3} relies on certain recurrence relations of $S_n(i)$. In a personal communication between the author and Andrews, the following question was raised:

\begin{question}
Is there a combinatorial interpretation of \eqref{eq:AY3}?
\end{question}

\noindent The purpose of this paper is to give an affirmative answer.

\section{The combinatorial interpretation}

We first notice that to interpret $S_n(1)$ using the language of partition theory, we need a weighted partition, which comes from the numerator term $(q;q)_{n+s}$ in each summand. This makes the interpretation less direct.
Hence we may slightly rewrite \eqref{eq:AY3} by multiplying both sides by $(-q;q)_n/(q;q)_n$.

\begin{theorem}\label{th:2.1}
We have
\begin{equation}\label{eq:2.1}
\sum_{s=0}^n q^s (-q^{s+1};q)_{n-s} \qbinom{n+s}{s}_q =(-q;q)_n^2. 
\end{equation}
\end{theorem}

Here
$$\qbinom{n+s}{s}_q:=\frac{(q;q)_{n+s}}{(q;q)_n (q;q)_s}$$
is the $q$-binomial coefficient.

To prove Theorem \ref{th:2.1}, we first need the following identity.

\begin{lemma}\label{le:2.2}
We have
\begin{equation}
\sum_{s=0}^n q^s (-q^{s+1};q)_{n-s} \qbinom{n+s}{s}_q = \sum_{t=0}^n q^{\binom{t+1}{2}} \qbinom{2n+1}{n+1+t}_q. 
\end{equation}
\end{lemma}

\begin{proof}
Let $\mathcal{B}_1$ be the set of partition pairs $(\lambda,\pi)$ such that $\lambda$ is a distinct partition (which can be empty) with largest part $\le n$ and $\pi$ is a partition with at most $n+1$ parts and largest part less than the smallest part of $\lambda$ (if $\lambda$ is empty, we assume that the largest part of $\pi$ is at most $n$). Then
$$\sum_{(\lambda,\pi)\in \mathcal{B}_1} q^{|\lambda|+|\pi|}=\sum_{s=0}^n q^s (-q^{s+1};q)_{n-s} \qbinom{n+s}{s}_q.$$
Here $|\lambda|$ means the sum of all parts of $\lambda$. In the sequel, one may also use the notation $|(\lambda,\pi)|$ to denote $|\lambda|+|\pi|$ for a partition pair $(\lambda,\pi)$. We further assume that $\ell=\ell(\lambda)$ counts the number of parts of $\lambda$, and $\lambda_1$ and $\lambda_\ell$ represent respectively the largest and smallest part of $\lambda$.

We also denote by $\mathcal{B}_2$ the set of partition pairs $(\mu,\nu)$ such that $\mu$ is a distinct partition with its parts being the first $t$ consecutive positive integers for some $0\le t\le n$ (i.e.~$\mu=(t,t-1,\ldots,2,1)$) and $\nu$ is a partition with at most $n+1+t$ parts and largest part $\le n-t$. Then
$$\sum_{(\mu,\nu)\in \mathcal{B}_2} q^{|\mu|+|\nu|}=\sum_{t=0}^n q^{\binom{t+1}{2}} \qbinom{2n+1}{n+1+t}_q.$$

Now consider the following map $\phi: \mathcal{B}_1\to \mathcal{B}_2$ with $\phi((\lambda,\pi))=(\mu,\nu)$ given by
\begin{enumerate}[\noindent (1)]
\item taking out $\ell(\lambda)$ from $\lambda_1$, $\ell(\lambda)-1$ from $\lambda_2$, ..., and $1$ from $\lambda_\ell$ to form $\mu$ (notice that $0\le \ell(\lambda)\le n$);
\item constructing a new partition $\lambda^*$ by $\lambda^*_i=\lambda_i-(\ell(\lambda)+1-i)$ for $1\le i\le \ell$ (notice that if the largest part of $\pi$ is $s$, then $s\le \lambda^*_\ell\le \cdots\le \lambda^*_1\le n-\ell(\lambda)$);
\item appending $\pi$ to the bottom of $\lambda^*$ to form $\nu$ (notice that this $\nu$ has at most $n+1+\ell(\lambda)$ parts with largest part $\le n-\ell(\lambda)$).
\end{enumerate}
For example, when $n=5$, if $\lambda=(5,3)$ and $\pi=(2,2,2,1,1)$, then $\phi((\lambda,\pi))=(\mu,\nu)$ with $\mu=(2,1)$ and $\nu=(3,2,2,2,2,1,1)$.

Notice that this map is well-defined and weight-preserving (viz.~$|\phi((\lambda,\pi))|=|(\lambda,\pi)|$). It is also easy to check that the map is invertible, and hence it is a bijection. This proves the lemma.
\end{proof}

\begin{remark}
Notice that the limiting case $q\to 1$ tells us
\begin{equation}\label{eq:1:q=1}
\sum_{s=0}^n 2^{n-s}\binom{n+s}{s} = \sum_{t=0}^n \binom{2n+1}{n+1+t}.
\end{equation}
This identity has a succinct combinatorial proof. We first consider the set $\mathcal{S}=\{1,2,\ldots,2n+1\}$. The r.h.s.~of \eqref{eq:1:q=1} counts the number of subsets of $\mathcal{S}$ with at least $n+1$ elements. On the other hand, for each subset of $\mathcal{S}$, we assume that the elements are in increasing order. Then the number of subsets such that the $(n+1)$th element is $n+1+s$ for some $0\le s\le n$ is
$$2^{n-s}\binom{n+s}{n}=2^{n-s}\binom{n+s}{s}.$$
Hence \eqref{eq:1:q=1} follows. We further notice that the complement of a subset of $\mathcal{S}$ with at least $n+1$ elements in $\mathcal{S}$ is a subset of $\mathcal{S}$ with at most $n$ elements. Hence subsets with at least $n+1$ elements attain half of $2^{\mathcal{S}}$. It follows that
$$\sum_{t=0}^n \binom{2n+1}{n+1+t}=\frac{2^{2n+1}}{2}=4^n,$$
which proves the limiting case ($q\to 1$) of \eqref{eq:2.1}.
\end{remark}

It remains to prove
\begin{equation}\label{eq:middle-step}
\sum_{t=0}^n q^{\binom{t+1}{2}} \qbinom{2n+1}{n+1+t}_q = (-q;q)_n^2.
\end{equation}
One may obtain an analytic proof by taking $z=-q^{-n}$ and $N=2n+1$ in the following $q$-binomial theorem (cf.~\cite[p.~36, Theorem 3.3]{And1976}):
$$(z;q)_N = \sum_{t=0}^N \qbinom{N}{t}(-1)^t z^t q^{\binom{t}{2}}.$$

However, since the purpose of this paper is to interpret Theorem \ref{th:2.1} combinatorially, we provide the following proof.

\begin{proof}[Combinatorial proof of \eqref{eq:middle-step}]
Again we need to slightly rewrite \eqref{eq:middle-step} by multiplying both sides by $q^{-n(n+1)/2}$:
$$\sum_{t=0}^n q^{\binom{t+1}{2}} q^{-\binom{n+1}{2}} \qbinom{2n+1}{n+1+t}_q = q^{-\binom{n+1}{2}} (-q;q)_n^2.$$

Now we consider a generalized distinct partition with parts in
$$\mathcal{N}=\{-n,-n+1,\ldots,-1,0,1,\ldots,n-1,n\}.$$
Let $\mathcal{P}$ be the set of such distinct partitions. We define the following two operators for $\lambda=\{\lambda_1,\ldots,\lambda_\ell\}\in\mathcal{P}$ (here we write $\lambda$ in set form):
\begin{enumerate}[\noindent (1)]
\item additive inverse: $-\lambda=\{-\lambda_1,\ldots,-\lambda_\ell\}$;
\item complement: if $\lambda\subset \pi\subset \mathcal{N}$, then the complement of $\lambda$ in $\pi$ is $\pi\backslash\lambda$ as set complement.
\end{enumerate}
Notice that the additive inverse and the complement in $\mathcal{N}$ are bijections from $\mathcal{P}$ to itself.

Let $\mu\in\mathcal{P}$ be either an empty partition or a distinct partition with merely negative parts. We have a bijection $\psi$ to the set of distinct partitions with merely positive parts (empty partition included) given by
$$\psi(\mu)=-(\{-n,-n+1,\ldots,-1\}\backslash \mu).$$
We further have
$$q^{|\mu|}=q^{-\binom{n+1}{2}}q^{|\psi(\mu)|}.$$
At last, for $\lambda\in\mathcal{P}$, whether $0$ is a part of $\lambda$ or not does not affect $|\lambda|$. Hence we get
$$\sum_{\lambda\in\mathcal{P}}q^{|\lambda|}=2q^{-\binom{n+1}{2}}(-q;q)_n^2.$$

We next consider $\lambda\in\mathcal{P}$ with at least $n+1$ parts. Denote by $\mathcal{P}_>$ the set of such partitions. We have a bijection $\tau$ between $\mathcal{P}_>$ and the set of distinct partitions in $\mathcal{P}$ with at most $n$ parts given by
$$\tau(\lambda)=-(\mathcal{N}\backslash \lambda).$$
Also we have
$$q^{|\lambda|}=q^{|\tau(\lambda)|}.$$
Hence
\begin{equation}\label{eq:half}
\sum_{\lambda\in\mathcal{P}_>}q^{|\lambda|}=\frac{1}{2}\sum_{\lambda\in\mathcal{P}}q^{|\lambda|}=q^{-\binom{n+1}{2}}(-q;q)_n^2.
\end{equation}

To prove \eqref{eq:middle-step}, we count the l.h.s.~of \eqref{eq:half} in a different way. Again suppose that $\lambda=\{\lambda_1,\lambda_2,\ldots,\lambda_{n+1+t}\}\in\mathcal{P}_>$ has $n+1+t$ parts for some $0\le t\le n$ with $-n\le \lambda_1<\lambda_2<\cdots<\lambda_{n+1+t}\le n$.

We denote by $\mathcal{B}_3$ the set of partition pairs $(\mu,\nu)$ such that $\mu$ is a distinct partition with its parts being the $n+1+t$ consecutive integers starting at $-n$ for some $0\le t\le n$ (i.e.~$\mu=\{-n,-n+1,\ldots,0,1,\ldots,t\}$) and $\nu$ is an ordinary partition with at most $n+1+t$ parts and largest part $\le n-t$.

We now construct a bijection $\rho$ between $\mathcal{P}_>$ and $\mathcal{B}_3$ by
\begin{enumerate}[\noindent (1)]
\item taking out $-n$ from $\lambda_1$, $-n+1$ from $\lambda_2$, ..., and $t$ from $\lambda_{n+1+t}$ to form $\mu$ (notice that $0\le t\le n$);
\item putting $\nu_i=\lambda_i-(-n+i-1)$ for $1\le i\le n+1+t$ (notice that $0\le \nu_1\le \nu_2\le \cdots\le \nu_{n+1+t}\le n-t$).
\end{enumerate}
For example, when $n=5$, if $\lambda=\{-4,-2,-1,0,2,4,5\}$, then $\rho(\lambda)=(\mu,\nu)$ with $\mu=\{-5,-4,-3,-2,-1,0,1\}$ and $\nu=\{1,2,2,2,3,4,4\}$.

Again it is easy to check that $\rho$ is well-defined, weight-preserving (viz.~$|\rho(\lambda)|=|\lambda|$) and invertible. Hence we have
\begin{equation}
\sum_{\lambda\in\mathcal{P}_>}q^{|\lambda|}=\sum_{t=0}^n q^{\binom{t+1}{2}} q^{-\binom{n+1}{2}} \qbinom{2n+1}{n+1+t}_q.
\end{equation}

Together with \eqref{eq:half}, we complete the proof.
\end{proof}

\section{Further remarks}

At the end of \cite{AY2017}, Andrews and Yee asked for bijective proofs of \eqref{eq:AY1} and \eqref{eq:AY2}. We notice that their proofs only rely on certain easy $q$-series manipulations as well as the sum $S_n(i)$. In this sense, most steps in their proofs can be interpreted combinatorially. However, we also notice that in several ``middle steps,'' both sides of the identities are being multiplied by certain functions of $q$. Due to this, the direct combinatorial interpretations of \eqref{eq:AY1} and \eqref{eq:AY2} are still fuzzy.

On the other hand, Andrews and Yee asserted that the combinatorial proofs of the following three identities are not difficult:
\begin{align}
\sum_{n\ge 0}\frac{z^n q^{2n^2+2n}}{(q;q^2)_{n+1} (zq;q^2)_{n+1}} &= \sum_{n\ge 0} \frac{z^n q^n}{(q;q^2)_{n+1}},\label{eq:omega}\\
\sum_{n\ge 0}\frac{q^{n^2+n}}{(-zq;q^2)_{n+1}}&= \sum_{n\ge 0} (q/z;q^2)_n (-zq)^n,\label{eq:nu1}\\
\sum_{n\ge 0}\frac{z^n q^{n^2+n}}{(-q;q^2)_{n+1}}&= \sum_{n\ge 0} (zq;q^2)_n (-q)^n.\label{eq:nu2}
\end{align}
However, it is still worth finishing these easy proofs to make their argument complete.

The proof of \eqref{eq:omega} is relatively routine. We first take $q\to q^2$, $z\to z^2$ in \eqref{eq:omega} and multiply by $zq$ on both sides:
\begin{equation}\label{eq:omega-1}
\sum_{n\ge 0}\frac{z^{2n+1} q^{(2n+1)^2}}{(q^2;q^4)_{n+1} (z^2q^2;q^4)_{n+1}} = \sum_{n\ge 0} \frac{z^{2n+1} q^{2n+1}}{(q^2;q^4)_{n+1}}.
\end{equation}

\begin{proof}[Combinatorial proof of \eqref{eq:omega-1}]
Recall that the Durfee square of a partition is the largest square that is contained within the partition's Ferrers diagram.

Now we consider partitions with the size of its Durfee square being an odd number and both the parts below its Durfee square and the conjugate of the parts to the right of its Durfee square forming an odd partition where each different part occurs an even number of times. It is easy to see that the largest part of such partition is an odd number. Let $\mathcal{DS}_k$ denote the set of such partitions with the largest part being $2k+1$. One readily writes the generating function
$$\sum_{k\ge 0}\sum_{\lambda\in \mathcal{DS}_k}z^{2k+1} q^{|\lambda|}=\sum_{n\ge 0}\frac{z^{2n+1} q^{(2n+1)^2}}{(q^2;q^4)_{n+1} (z^2q^2;q^4)_{n+1}}.$$

On the other hand, let $\mathcal{OE}_k$ be the set of partition pairs $(\mu,\nu)$ where $\mu$ has only one odd part $2k+1$ and $\nu$ is an odd partition with largest part $\le 2k+1$ where each different part occurs an even number of times. We have
$$\sum_{k\ge 0}\sum_{(\mu,\nu)\in \mathcal{OE}_k}z^{2k+1} q^{|\mu|+|\nu|}=\sum_{n\ge 0} \frac{z^{2n+1} q^{2n+1}}{(q^2;q^4)_{n+1}}.$$

There is a trivial bijection between $\mathcal{DS}_k$ and $\mathcal{OE}_k$. Given $\lambda\in \mathcal{DS}_k$, we split the largest part and the remaining parts as $\mu$ and $\nu$ respectively. Then $(\mu,\nu)\in \mathcal{OE}_k$. Inversely, if $(\mu,\nu)\in \mathcal{OE}_k$, we attach $\mu$ to the top of $\nu$ to get a partition $\lambda\in \mathcal{DS}_k$. To see this, we only need to show that the size of the Durfee square of $\lambda$ is an odd number. Assume not, then the largest part $\lambda_i$ below the Durfee square is strictly smaller than the part above $\lambda_i$. Also, there is an odd number of parts below the Durfee square. Hence there exists a part of $\nu$ that occurs an odd number of times, leading to a contradiction.

The desired identity follows directly from this bijection.
\end{proof}

\begin{example}
Figure \ref{Figure 0} illustrates the bijection in the proof of \eqref{eq:omega-1}.

\begin{figure}[ht]
\caption{Bijection in the proof of \eqref{eq:omega-1}}
\vspace{1em}
\label{Figure 0}
\includegraphics[width=0.8\textwidth]{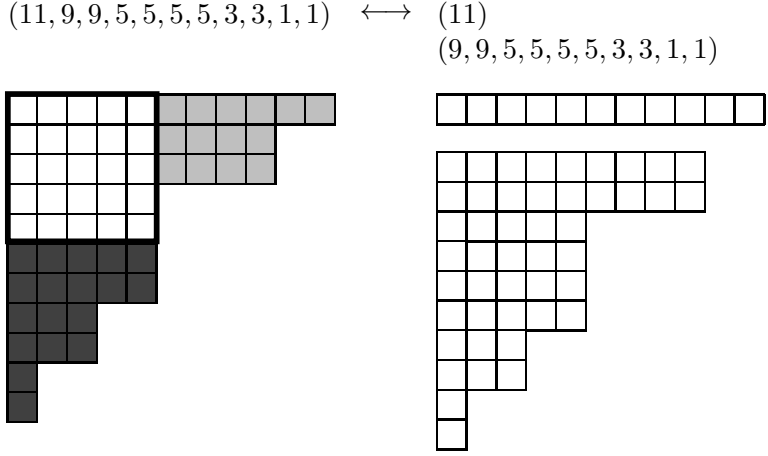}
\end{figure}
\end{example}

The proof of \eqref{eq:nu1} and \eqref{eq:nu2} is, in some sense, more intriguing. In fact, we will prove the following identity:
\begin{equation}\label{eq:nu3}
\sum_{n\ge 0}\frac{q^{n^2+n}x^n}{(yq;q^2)_{n+1}}= \sum_{n\ge 0} (-xq/y;q^2)_n (yq)^n,
\end{equation}
which is derived by taking $x\to xq$ and $y\to yq$ in \cite[p.~29, Example 6]{And1976}:
$$\sum_{n\ge 0}\frac{q^{n^2}x^n}{(y;q^2)_{n+1}}= \sum_{n\ge 0} (-xq/y;q^2)_n y^n.$$

\begin{proof}[Combinatorial proof of \eqref{eq:nu3}]
Let $\mathcal{O}_{n,k}$ denote the set of partition pairs $(\lambda,\pi)$ where $\lambda$ is a $n\times (n+1)$ rectangle (i.e.~$\lambda$ has $n+1$ parts and each part is $n$) and $\pi$ is an odd partition with exactly $k$ parts where the largest part is at most $2n+1$. We have
$$\sum_{n\ge 0}\sum_{k\ge 0}\sum_{(\lambda,\pi)\in \mathcal{O}_{n,k}}x^n y^k q^{|\lambda|+|\pi|}=\sum_{n\ge 0}\frac{q^{n^2+n}x^n}{(yq;q^2)_{n+1}}.$$

On the other hand, let $\mathcal{DO}_{n,k}$ be the set of partition pairs $(\mu,\nu)$ where $\mu$ has only one part $n+k$ and $\nu$ is a distinct odd partition with exactly $n$ parts where the largest part is at most $2(n+k)-1$. One may write the generating function
$$\sum_{n\ge 0}\sum_{k\ge 0}\sum_{(\mu,\nu)\in \mathcal{DO}_{n,k}}x^n y^k q^{|\mu|+|\nu|}=\sum_{n\ge 0} (-xq/y;q^2)_n (yq)^n,$$
by enumerating the size of $\mu$.

We next construct a bijection between $\mathcal{O}_{n,k}$ and $\mathcal{DO}_{n,k}$.

Given an arbitrary $(\lambda,\pi)\in \mathcal{O}_{n,k}$. We split the largest part of $\pi$, say $2s+1$ ($0\le s\le n$), into two pieces, one of which is $s+1$, and the other is $s$. We now append $s+1$ to the right of $\lambda$ and $s$ below $\lambda$. Repeating the same process for all parts of $\pi$, then we get a unique partition $\nu^*$. Notice that the largest part of $\nu^*$ equals $n+k$.

We take out the largest part of $\nu^*$ to form $\mu$. The remaining parts form a new partition $\nu'$. There are three trivial observations:
\begin{enumerate}[\noindent (1)]
\item $\nu'$ is a self-conjugate partition (i.e.~the conjugate of its Ferrers diagram remains the same as its Ferrers diagram);
\item the size of the Durfee square of $\nu'$ is $n$;
\item the largest part of $\nu'$ is at most $n+k$.
\end{enumerate}

At last, we recall that there is a bijection between self-conjugate partitions with the size of its Durfee square being $n$ and distinct odd partitions with exactly $n$ parts. Hence, we can transform $\nu'$ into $\nu$, a distinct odd partition with exactly $n$ parts where the largest part is at most $2(n+k)-1$.

Clearly, we have $(\mu,\nu)\in \mathcal{DO}_{n,k}$. Notice that the above process is invertible, and hence we obtain the bijection.

The desired identity follows directly by comparing the generating functions.
\end{proof}

\begin{example}
Figure \ref{Figure 1} illustrates the bijection in the proof of \eqref{eq:nu3}.

\begin{figure}[ht]
\caption{Bijection in the proof of \eqref{eq:nu3}}
\vspace{1em}
\label{Figure 1}
\includegraphics[width=0.8\textwidth]{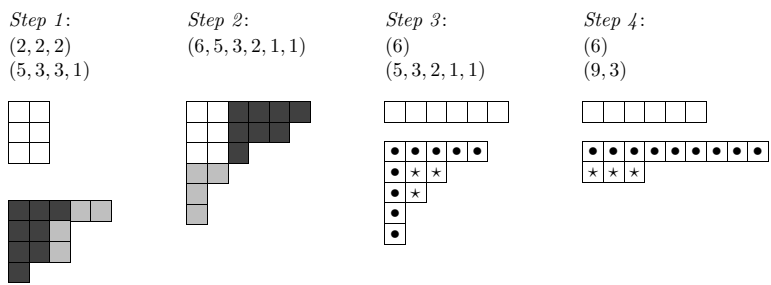}
\end{figure}
\end{example}

\subsection*{Acknowledgements}

I would like to thank George E. Andrews and Ae Ja Yee for many helpful discussions. I also want to thank the referee for the careful reading and helpful comments.

\bibliographystyle{amsplain}

\end{document}